\documentclass[11pt]{amsart}
\usepackage{amssymb}
\usepackage{amsmath}
\usepackage{mathrsfs}
\usepackage{amsthm}
\usepackage{enumerate}
\newtheorem{theorem}{Theorem}[section]
\newtheorem{lemma}[theorem]{Lemma}
\theoremstyle{definition}
\newtheorem{problem}{Problem}

\numberwithin{equation}{section}
\begin{document}
\title{Multiplication operators on $L^{p}$}
\author{March T.~Boedihardjo}
\address{Department of Mathematics, University of California, Los Angeles, CA 90095-1555}
\email{march@math.ucla.edu}
%\keywords{Multiplication operator, Lp space, Approximate similarity}
%\subjclass[2010]{47B38}
\begin{abstract}
We show that every operator on $L^{p}$, $1<p<\infty$ defined by multiplication by the identity function on $\mathbb{C}$ is a compact perturbation of an operator that is diagonal with respect to an unconditional basis. We also classify these operators up to similarity modulo compact operators and up to approximate similarity.
\end{abstract}
\maketitle
\section{Introduction}
Let $\mu_{1},\mu_{2}$ be measures on $\mathbb{C}$ with compact supports. Consider the operator $M_{\mu_{i}}$ on $L^{2}(\mu_{i})$ defined by
\begin{equation}\label{multiplication}
(M_{\mu_{i}}f)(z)=zf(z),\quad z\in\mathbb{C},\,f\in L^{2}(\mu_{i}),\,i=1,2.
\end{equation}
It is well known \cite[Theorem 4.58]{Douglas} that $M_{\mu_{1}}$ and $M_{\mu_{2}}$ are unitarily equivalent if and only if $\mu_{1}$ and $\mu_{2}$ are equivalent (i.e., absolutely continuous with respect to each other).

The situation is quite different if we allow a compact perturbation. A classical result of Berg \cite{Berg} states that every normal operator on a separable Hilbert space is the sum of a diagonal operator and a compact operator. As a consequence, $M_{\mu_{1}}$ is unitarily equivalent to a compact perturbation of $M_{\mu_{2}}$ if and only if $\mathrm{supp}(\mu_{1})$ and $\mathrm{supp}(\mu_{2})$ have the same cluster points \cite[Corollary 2.13]{Pearcy}.

In this paper, we consider the question of whether these results extend to $L^{p}$ for $1<p<\infty$. If $\mu$ is a measure on $\mathbb{C}$ with compact support, we will use the same notation $M_{\mu}$ for the operator on $L^{p}(\mu)$ defined as in (\ref{multiplication}). Our first main result is
\begin{theorem}\label{main1}
Let $1<p<\infty$. Let $\mu$ be a finite measure on $\mathbb{C}$ with compact support. Then for every $\epsilon>0$, there exist a compact operator $K$ on $L^{p}(\mu)$ with $\|K\|<\epsilon$ and an operator $D$ on $L^{p}(\mu)$ that is diagonal with respect to a unconditional basis $(u_{n})_{n\geq 1}$ for $L^{p}(\mu)$ such that
\[M_{\mu}=D+K,\]
and the unconditional basis constant of $(u_{n})_{n\geq 1}$ is at most $C(p)$ where $C(p)>0$ depends only on $p$.
\end{theorem}
Note that the unconditional basis $(u_{n})_{n\geq 1}$ depends on $\epsilon$.

Let $X_{1}$ and $X_{2}$ be Banach spaces. An operator $T_{1}$ on $X_{1}$ and an operator $T_{2}$ on $X_{2}$ are {\it similar} if there is an isomorphism $W:X_{1}\to X_{2}$ such that $T_{2}=WT_{1}W^{-1}$.

Unlike for Hilbert space, not all unconditional bases for $L^{p}(\mu)$, $1<p<\infty$ are equivalent. Thus Theorem \ref{main1} does not give similarity modulo compact operators of different $M_{\mu}$. In fact the characterization of similarity of $M_{\mu}$ modulo compact operators for $L^{p}$ is quite different from $L^{2}$. Our second main result is
\begin{theorem}\label{main2}
Suppose that $1<p<\infty$, $p\neq 2$. Let $\mu_{1}$ and $\mu_{2}$ be finite measures on $\mathbb{C}$ with compact supports. Assume that the support of $\mu_{1}$ and the support of $\mu_{2}$ are not finite sets. Let $\mu_{1}'$ and $\mu_{2}'$ be the purely nonatomic parts of $\mu_{1}$ and $\mu_{2}$, respectively. Then the following statements are equivalent.
\begin{enumerate}[(i)]
\item There exists a compact operator $K$ on $L^{p}(\mu_{2})$ such that $M_{\mu_{1}}$ is similar to $M_{\mu_{2}}+K$.
\item $\mu_{1}'$ and $\mu_{2}'$ are equivalent and the sets $\mathrm{supp}(\mu_{1})$ and $\mathrm{supp}(\mu_{2})$ have the same cluster points.
\end{enumerate}
\end{theorem}
Let us compare Theorem \ref{main2} with the characterization when $p=2$. Suppose that $\mu_{1}$ and $\mu_{2}$ are purely nonatomic and mutually singular but have the same support. Then $M_{\mu_{1}}$ and $M_{\mu_{2}}$ are unitarily equivalent modulo compact operators when $p=2$ but are not similar modulo compact operators when $p\neq 2$.

On the other hand, instead, suppose that $\mu_{1}$ is the Lebesgue measure on $[0,1]$ and $\mu_{2}$ is the sum of the Lebesgue measure on $[0,1]$ and an (atomic) measure on the rational numbers in $[0,1]$. Then $M_{\mu_{1}}$ and $M_{\mu_{2}}$ are unitarily equivalent modulo compact operators when $p=2$ and are also similar modulo compact operators when $p\neq 2$. But in both cases $p=2$ and $p\neq 2$, the operators $M_{\mu_{1}}$ and $M_{\mu_{2}}$ are not similar (by considering the existence of eigenvalues) even though they have the same spectrum and essential spectrum.

In Section 2, we give the proof of Theorem \ref{main1} using generalized Haar systems. We also show that when $\mu$ is supported on $\mathbb{R}$, the compact operator $K$ in Theorem \ref{main1} can be chosen to be 1-summing. When $p=2$, this was shown by von Neumann \cite[Theorem 2.10]{Pearcy}. (1-summing operators on Hilbert space are precisely the Hilbert-Schmidt operators \cite{Pelczynski}.) In Section 3, we prove Theorem \ref{main2} using a technique from \cite{Korotov}. As a by-product, we also classify $M_{\mu}$ up to approximate similarity. In Section 4, we state a few open problems.

We begin by introducing some notation and terminology that are needed in what follows.

If $a\in\mathbb{C}$ and $r>0$, the square on the complex plane with center $a$ and side length $r$ is denoted by $S(a,r)$, i.e.,
\[S(a,r)=\left\{x+iy\in\mathbb{C}:|x-\mathrm{Re}\,a|\leq\frac{r}{2}\text{ and }|y-\mathrm{Im}\,a|\leq\frac{r}{2}\right\}.\]
If $A$ is a subset of $\mathbb{C}$, the indicator function of $A$ is denoted by $I(A)$; the diameter of $A$ is denoted by $\mathrm{diam}(A)$; and the complement of $A$ in $\mathbb{C}$ is denoted by $A^{c}$. If $\mu$ is a measure on $\mathbb{C}$ and $A$ is a Borel set in $\mathbb{C}$ then $P_{A}^{\mu}$ is the operator on $L^{p}(\mu)$ defined by
\[(P_{A}^{\mu}f)(z)=I(A)(z)f(z),\quad f\in L^{p}(\mu),\,z\in\mathbb{C}.\]
Let $X$ be a Banach space. A sequence $(u_{n})_{n\geq 1}$ is a $C$-{\it unconditional basis} for $X$ if
\[\frac{1}{C}\left\|\sum_{n=1}^{m}|a_{n}|u_{n}\right\|\leq\left\|\sum_{n=1}^{m}a_{n}u_{n}\right\|\leq C\left\|\sum_{n=1}^{m}|a_{n}|u_{n}\right\|,\]
for every $m\geq 1$ and $a_{1},\ldots,a_{m}\in\mathbb{C}$.

If $S$ is a subset of $X$, the closed linear span of $S$ in $X$ is denoted by $\vee S$. The space of operators on $X$ is denoted by $B(X)$.

If $X_{1}$ and $X_{2}$ are Banach spaces, two operators $T_{1}\in B(X_{1})$ and $T_{2}\in B(X_{2})$ are {\it approximately similar} \cite{Hadwin} if there is a sequence $(W_{n})_{n\geq 1}$ of isomorphisms from $X_{1}$ onto $X_{2}$ such that
\[\sup_{n\geq 1}\|W_{n}\|\|W_{n}^{-1}\|<\infty,\]
\[T_{2}-W_{n}T_{1}W_{n}^{-1}\text{ is compact for all }n\geq 1,\]
\[\lim_{n\to\infty}\|T_{2}-W_{n}T_{1}W_{n}^{-1}\|=0.\]
If $T$ is an operator from $X_{1}$ into $X_{2}$, the 1-summing norm (see, e.g., \cite{Diestel}) of $T$ is denoted by $\pi_{1}(T)$.

When we say a ``diagonal operator on $l^{p}$,'' we always mean an operator on $l^{p}$ that is diagonal with respect to the canonical basis for $l^{p}$. All measures are assumed to be finite. Throughout this paper, $1<p<\infty$ and $C(p)$ is a positive constant that depends only on $p$.
\section{Proof of Theorem \ref{main1}}
We begin by recalling from \cite{Dor} the definition of generalized Haar system. A system $(A_{n,j}:n=0,1,2,\ldots,\,1\leq j\leq 2^{n})$ of measurable sets is a {\it dyadic tree} if
\[A_{n+1,2j-1}\cap A_{n+1,2j}=\emptyset,\]
and
\[A_{n+1,2j-1}\cup A_{n+1,2j}=A_{n,j},\]
for all $n\geq 0$ and $1\leq j\leq 2^{n}$. Let
\[\Lambda=\{(n,j):\,n\geq 1,1\leq j\leq 2^{n-1},\mu(A_{n,2j-1})>0\text{ and }\mu(A_{n,2j})>0\}\cup\{(0,1)\}.\]
Let $\mu$ be a purely nonatomic measure. The {\it generalized Haar system} $\{h_{n,j}:(n,j)\in\Lambda\}$ with respect to $(A_{n,j})$ is defined as follows:
\[h_{0,1}=I(A_{0,1})/\|I(A_{0,1})\|_{L^{p}(\mu)}\]
and
\[h_{n,j}=H_{n,j}/\|H_{n,j}\|_{L^{p}(\mu)},\]
where
\[H_{n,j}=I(A_{n,2j-1})/\mu(A_{n,2j-1})-I(A_{n,2j})/\mu(A_{n,2j}),\quad n\geq 1,\,(n,j)\in\Lambda.\]
Note that when $(n,j)\notin\Lambda$, at least one of $\mu(A_{n,2j-1}),\mu(A_{n,2j-1})$ vanishes so both of $I(A_{n,2j-1})$ and $I(A_{n,2j})$ are already in the span of some $h_{m,k}$ where $m<n$. To see this, observe that if $\mu(A_{n,2j-1})=0$ then $I(A_{n,2j})=I(A_{n-1,j})$ which is in the span of some $h_{m,k}$ for some $m<n$.
\begin{lemma}[\cite{Dor}]\label{unconditional}
The system $\{h_{n,j}:(n,j)\in\Lambda\}$ defines a normalized $C(p)$-unconditional basis for $L^{p}(\mathscr{B},\mu)$ where $\mathscr{B}$ is the $\sigma$-algebra generated by the $A_{n,j}$.
\end{lemma}
\begin{lemma}\label{square}
There exists a dyadic tree $(A_{n,j}:n=0,1,2,\ldots,\,1\leq j\leq 2^{n})$ such that each $A_{n,j}$ is in $S(0,1)$ and
\begin{equation}\label{diameter}
\mathrm{diam}(A_{n,j})\leq\frac{1}{2^{\frac{n}{2}-1}},\quad n\geq 0,\,1\leq j\leq 2^{n},
\end{equation}
and the $\sigma$-algebra generated by the $A_{n,j}$ is the Borel $\sigma$-algebra on $S(0,1)$.
\end{lemma}
\begin{proof}
Take $A_{0,1}=S(0,1)$. Each $A_{n,j}$ will be defined as rectangles. Let $x_{n,j}+iy_{n,j}$ be the center of $A_{n,j}$. If $n\geq 0$ is even, take
\[A_{n+1,2j-1}=A_{n,j}\cap\{x+iy\in S(0,1):y>y_{n,j}\},\]
and
\[A_{n+1,2j}=A_{n,j}\cap\{x+iy\in S(0,1):y\leq y_{n,j}\}.\]
If $n$ is odd, take
\[A_{n+1,2j-1}=A_{n,j}\cap\{x+iy\in S(0,1):x>x_{n,j}\},\]
and
\[A_{n+1,2j}=A_{n,j}\cap\{x+iy\in S(0,1):x\leq x_{n,j}\}.\]
Then $A_{n,j}$ satisfies the required properties.
\end{proof}
\begin{proof}[Proof of Theorem \ref{main1}]
Without loss of generality, we may assume that $\mu$ is purely nonatomic and has support in $S(0,1)$. Let $(A_{n,j}:n=0,1,2,\ldots,\,1\leq j\leq 2^{n})$ be defined as in Lemma \ref{square}. Let $\{h_{n,j}:(n,j)\in\Lambda\}$ be the generalized Haar system with respect to $(A_{n,j})$. By Lemma \ref{unconditional}, $\{h_{n,j}:(n,j)\in\Lambda\}$ is a $C(p)$-unconditional basis for $L^{p}(\mu)$. Observe that
\begin{equation}\label{samespan}
\vee\{h_{n,j}:(n,j)\in\Lambda,\,n\leq N\}=\vee\{I(A_{N,j}):1\leq j\leq 2^{N}\},\quad N\geq 0.
\end{equation}
Fix $N\geq 1$. We have that $\{h_{n,j}:(n,j)\in\Lambda,\,n\geq N+1\}\cup\{I(A_{N,j}):1\leq j\leq 2^{N}\}$ is a $C(p)$-unconditional basis for $L^{p}(\mu)$.

For every $n\geq 0$, $1\leq j\leq 2^{n}$, pick a point $\lambda_{n,j}$ in $A_{n,j}$. If $A_{n,j}=\emptyset$, take $\lambda_{n,j}=0$. Since $h_{n,j}$ is supported on $A_{n-1,j}$, by (\ref{diameter}), we have
\begin{equation}\label{eventualdiagapproxp}
\|M_{\mu}h_{n,j}-\lambda_{n-1,j}h_{n,j}\|\leq\frac{1}{2^{\frac{n-1}{2}-1}},\quad (n,j)\in\Lambda,
\end{equation}
and
\begin{equation}\label{inidiagapproxp}
\|M_{\mu}I(A_{N,j})-\lambda_{N,j}I(A_{N,j})\|\leq\frac{1}{2^{\frac{N}{2}-1}}\|I(A_{N,j})\|,\quad 1\leq j\leq 2^{N}.
\end{equation}
Take $D$ to be the operator on $L^{p}(\mu)$ that is diagonal with respect to the unconditional basis $\{h_{n,j}:(n,j)\in\Lambda,\,n\geq N+1\}\cup\{I(A_{N,j}):1\leq j\leq 2^{N}\}$ and with entries $\lambda_{n-1,j}$, $(n,j)\in\Lambda$, $n\geq N+1$ and $\lambda_{N,j}$, $1\leq j\leq 2^{N}$.

Note that for each $n\geq 0$, all the $h_{n,j}$ have disjoint supports and all the $M_{\mu}h_{n,j}-\lambda_{n,j}h_{n,j}$ have disjoint supports. It follows that
\begin{equation}\label{eventualdiagapproxs}
\|M_{\mu}x-Dx\|\leq\frac{1}{2^{\frac{n-1}{2}-1}}\|x\|,\quad x\in\vee\{h_{n,j}:1\leq j\leq 2^{n-1},\,(n,j)\in\Lambda\},\,n\geq N+1.
\end{equation}
Similarly we have
\begin{equation}\label{inidiagapproxs}
\|M_{\mu}x-Dx\|\leq\frac{1}{2^{\frac{N}{2}-1}}\|x\|,\quad x\in\vee\{I(A_{N,j}):1\leq j\leq 2^{N}\}.
\end{equation}
For each $m\geq 0$, let $P_{m}$ be the projection from $L^{p}(\mu)$ onto $\vee\{h_{m,j}:1\leq j\leq 2^{m-1},\,(m,j)\in\Lambda\}$ when $m\geq 1$ or $\vee\{h_{0,1}\}$ when $m=0$. Since $\{h_{n,j}:(n,j)\in\Lambda\}$ has unconditional basis constant at most $C(p)$, we have $\|P_{m}\|\leq C(p)$ and $\|P_{0}+\ldots+P_{m}\|\leq C(p)$ for all $m\geq 0$. Thus by (\ref{eventualdiagapproxs}),
\begin{equation}\label{aboveN}
\|(M_{\mu}-D)P_{n}\|\leq\frac{C(p)}{2^{\frac{n-1}{2}-1}},\quad n\geq N+1.
\end{equation}
By (\ref{samespan}) and (\ref{inidiagapproxs}), we have
\begin{equation}\label{belowN}
\|(M_{\mu}-D)(P_{0}+\ldots+P_{N})\|\leq\frac{C(p)}{2^{\frac{N}{2}-1}}.
\end{equation}
Since $P_{0}+P_{1}+P_{2}+\ldots$ is the identity operator and each $P_{n}$ has finite rank, it follows that from (\ref{aboveN}) that $M_{\mu}-D$ is compact. From (\ref{aboveN}) and (\ref{belowN}), we have
\[\|M_{\mu}-D\|\leq\frac{C(p)}{2^{\frac{N}{2}-1}}+\sum_{n\geq N+1}\frac{C(p)}{2^{\frac{n-1}{2}-1}},\]
which is arbitarily small if $N$ is taken to be large enough. Thus the result follows.
\end{proof}
\begin{theorem}\label{main1abs}
Let $\mu$ be a measure on $\mathbb{R}$ with compact support. Then for every $\epsilon>0$, there exist a compact operator $K$ on $L^{p}(\mu)$ with $\pi_{1}(K)<\epsilon$ and an operator $D\in B(L^{p}(\mu))$ that is diagonal with respect to a $C(p)$-unconditional basis $(u_{n})_{n\geq 1}$ for $L^{p}(\mu)$ such that
\[M_{\mu}=D+K.\]
\end{theorem}
We need the following lemma to prove Theorem \ref{main1abs}. This lemma should be well known and can be proved using the techniques in \cite{Pelczynski}.
\begin{lemma}\label{diagabs}
Let $D_{0}$ be a diagonal operator on $l^{p}$ with entries of $a_{1},a_{2},\ldots$. Then
\[\pi_{1}(D_{0})\leq\begin{array}{cc}\begin{cases}C(p)\left(\sum|a_{n}|^{q}\right)^{\frac{1}{q}},&p\geq 2\\C(p)\left(\sum|a_{n}|^{2}\right)^{\frac{1}{2}},&p\leq 2\end{cases},\end{array}\]
where $\displaystyle\frac{1}{p}+\frac{1}{q}=1$.
\end{lemma}
We explain the changes we need to make on the proof of Theorem \ref{main1} to prove Theorem \ref{main1abs}. Since $\mu$ is supported on $\mathbb{R}$, we may assume that $\mu$ has support in $[0,1]$ and is purely nonatomic. Take $A_{n,j}$ to be the dyadic intervals. The length of $A_{n,j}$ is $\displaystyle\frac{1}{2^{n}}$. (\ref{eventualdiagapproxp}) and (\ref{inidiagapproxp}) can be strengthened and become
\[\|M_{\mu}h_{n,j}-\lambda_{n-1,j}h_{n,j}\|\leq\frac{1}{2^{n-1}},\quad (n,j)\in\Lambda,\]
and
\[\|M_{\mu}I(A_{N,j})-\lambda_{N,j}I(A_{N,j})\|\leq\frac{1}{2^{N}}\|I(A_{N,j})\|,\quad 1\leq j\leq 2^{N}.\]
The operator $D$ is defined in the same way. By Lemma \ref{diagabs}, we have
\[\pi_{1}((M_{\mu}-D)|_{\vee\{h_{n,j}:1\leq j\leq 2^{n-1},(n,j)\in\Lambda\}})\leq C(p)\cdot 2^{(n-1)/\min(q,2)}/2^{n-1},\quad n\geq N+1,\]
and
\[\pi_{1}((M_{\mu}-D)|_{\vee\{I(A_{N,j}):1\leq j\leq 2^{N}\}})\leq C(p)\cdot 2^{N/\min(q,2)}/2^{N}.\]
The rest is analogous to the proof of Theorem \ref{main1}.
\section{Proof of Theorem \ref{main2}}
\begin{lemma}\label{estimate}
If $a,b>0$ then
\[(a+b)^{p}\leq b^{p}+ap(a+b)^{p-1}.\]
\end{lemma}
\begin{proof}
\[(a+b)^{p}-b^{p}=\int_{b}^{a+b}pt^{p-1}\,dt\leq ap(a+b)^{p-1}.\]
\end{proof}
\begin{lemma}\label{local}
Let $\mu_{1}$ and $\mu_{2}$ be measures on $\mathbb{C}$ with compact supports. Let $L:L^{p}(\mu_{1})\to L^{p}(\mu_{2})$ be an operator such that $M_{\mu_{2}}L-LM_{\mu_{1}}$ is compact. Let $D\subset E\subset\mathbb{C}$ be such that $D$ is compact and $E$ is open. Then the operator $P_{D}^{\mu_{2}}LP_{E^{c}}^{\mu_{1}}:L^{p}(\mu_{1})\to L^{p}(\mu_{2})$ is compact.
\end{lemma}
\begin{proof}
Since $M_{\mu_{2}}L-LM_{\mu_{1}}$ is compact,
\begin{align*}
&(P_{D}^{\mu_{2}}M_{\mu_{2}}P_{D}^{\mu_{2}})(P_{D}^{\mu_{2}}LP_{E^{c}}^{\mu_{1}})-(P_{D}^{\mu_{2}}LP_{E^{c}}^{\mu_{1}})(P_{E^{c}}^{\mu_{1}}M_{\mu_{1}}P_{E^{c}}^{\mu_{1}})\\=&P_{D}^{\mu_{2}}(M_{\mu_{2}}L-LM_{\mu_{1}})P_{E^{c}}^{\mu_{1}}
\end{align*}
is compact. Note that the operators $P_{D}^{\mu_{2}}M_{\mu_{2}}P_{D}^{\mu_{2}}$ and $P_{E^{c}}^{\mu_{1}}M_{\mu_{1}}P_{E^{c}}^{\mu_{1}}$ have disjoint spectra (as operators on $P_{D}^{\mu_{2}}(L^{p}(\mu_{2}))$ and $P_{E^{c}}^{\mu_{1}}(L^{p}(\mu_{1}))$, respectively.) So by Rosenblum's Theorem (see \cite[Theorem 0.12]{Radjavi}), $P_{D}^{\mu_{2}}LP_{E^{c}}^{\mu_{1}}$ is compact.
\end{proof}
\begin{lemma}\label{singular}
Let $\mu_{1},\mu_{2}$ be a mutually singular measures on $\mathbb{C}$. Let $f\in L^{p}(\mu_{2})$. Then for every $\epsilon>0$, there exist $F\subset G\subset\mathbb{C}$ such that $F$ is compact, $G$ is open, $\mu_{1}(G)<\epsilon$ and $\|P_{F^{c}}^{\mu_{2}}f\|_{L^{p}(\mu_{2})}<\epsilon$.
\end{lemma}
\begin{proof}
There exists a Borel set $A$ in $\mathbb{C}$ such that $\mu_{1}(A)=0$ and $\mu_{2}(A^{c})=0$. Note that $\mu_{1}$ and $\mu_{2}$ are regular \cite[Theorem 2.18]{Rudin}. There exists an open set $G$ containing $A$ such that $\mu_{1}(G)<\epsilon$. For every $\delta>0$, there exists a compact subset $F$ of $G$ such that $\mu_{2}(G\backslash F)<\delta$. Since $A\subset G$, $G^{c}\subset A^{c}$ so $\mu_{2}(G^{c})=0$ so $\mu_{2}(F^{c})<\delta$. Choose $\delta>0$ small enough so that we have $\|P_{F^{c}}^{\mu_{2}}f\|_{L^{p}(\mu_{2})}<\epsilon$.
\end{proof}
\begin{lemma}\label{nonembed}
Assume that $p>2$. Let $\mu_{1}$ and $\mu_{2}$ be mutually singular measures on $\mathbb{C}$ with compact supports. Assume that $\mu_{1}$ is purely nonatmoic and nontrivial. Let $L:L^{p}(\mu_{1})\to L^{p}(\mu_{2})$ be an operator such that $M_{\mu_{2}}L-LM_{\mu_{1}}$ is compact. Then $L$ is not bounded below.
\end{lemma}
\begin{proof}
Since $\mu_{1}$ is purely nonatomic and nontrivial, there exist independent $f_{1},f_{2},\ldots\in L^{p}(\mu_{1})$ taking $\pm 1$ values.

Let $0<\epsilon<1$. Let $D_{1}=E_{1}=\emptyset$ and $n_{1}=1$. Suppose that $D_{1}\subset\ldots\subset D_{k}$, $E_{1}\subset\ldots\subset E_{k}$, $n_{1}<n_{2}<\ldots<n_{k}$ have been chosen so that $D_{i}$ is compact, $E_{i}$ is open and $D_{i}\subset E_{i}$ for all $i=1,\ldots,k$. By Lemma \ref{singular}, there exist subsets $F\subset G\subset\mathbb{C}$ such that $F$ is compact, $G$ is open, $\mu_{1}(G)<\epsilon$ and
\begin{equation}\label{support}
\|P_{F^{c}}^{\mu_{2}}LP_{E_{k}^{c}}^{\mu_{1}}f_{n_{k}}\|_{L^{p}(\mu_{2})}<\epsilon.
\end{equation}
Take $D_{k+1}=D_{k}\cup F$ and $E_{k+1}=E_{k}\cup G$. By Lemma \ref{local}, $P_{D_{k+1}}^{\mu_{2}}LP_{E_{k+1}^{c}}^{\mu_{1}}$ is compact. Since $f_{n}\to 0$ weakly, there exists $n_{k+1}>n_{k}$ such that
\begin{equation}\label{local3}
\|P_{D_{k+1}}^{\mu_{2}}LP_{E_{k+1}^{c}}^{\mu_{1}}f_{n_{k+1}}\|_{L^{p}(\mu_{2})}<\epsilon.
\end{equation}
Thus $D_{k+1},E_{k+1},n_{k+1}$ are defined.

By (\ref{support}),
\begin{equation}\label{support2}
\|P_{D_{k+1}^{c}}^{\mu_{2}}LP_{E_{k}^{c}}^{\mu_{1}}f_{n_{k}}\|_{L^{p}(\mu_{2})}<\epsilon,\quad k\geq 1.
\end{equation}
For each $k\geq 1$, let $g_{k}=P_{E_{k}^{c}}^{\mu_{1}}f_{n_{k}}$. By (\ref{local3}) and (\ref{support2}), for every integers $1\leq j<k$, we have
\[\|P_{D_{k}}^{\mu_{2}}Lg_{k}\|_{L^{p}(\mu_{2})}<\epsilon,\]
and
\[\|P_{D_{k}^{c}}^{\mu_{2}}Lg_{j}\|_{L^{p}(\mu_{2})}\leq\|P_{D_{j+1}^{c}}^{\mu_{2}}Lg_{j}\|_{L^{p}(\mu_{2})}<\epsilon.\]
So
\begin{eqnarray*}
\|Lg_{1}+\ldots+Lg_{k}\|_{L^{p}(\mu_{2})}&<&k\epsilon+\|P_{D_{k}}^{\mu_{2}}(Lg_{1}+\ldots+Lg_{k-1})+P_{D_{k}^{c}}^{\mu_{2}}Lg_{k}\|_{L^{p}(\mu_{2})}
\\&\leq&k\epsilon+(\|Lg_{1}+\ldots+Lg_{k-1}\|_{L^{p}(\mu_{2})}^{p}+\|Lg_{k}\|_{L^{p}(\mu_{2})}^{p})^{\frac{1}{p}}.
\end{eqnarray*}
So by Lemma \ref{estimate},
\[\|Lg_{1}+\ldots+Lg_{k}\|_{L^{p}(\mu_{2})}^{p}\leq
\|Lg_{1}+\ldots+Lg_{k-1}\|_{L^{p}(\mu_{2})}^{p}+\|Lg_{k}\|_{L^{p}(\mu_{2})}^{p}+C(p,k,\|L\|)\epsilon,\]
where $C(p,k,\|L\|)>0$ depends only on $p,k,\|L\|$ but not on $\epsilon$. By induction, we obtain
\[\|Lg_{1}+\ldots+Lg_{k}\|_{L^{p}(\mu_{2})}^{p}\leq\|Lg_{1}\|_{L^{p}(\mu_{2})}^{p}+\ldots+\|Lg_{k}\|_{L^{p}(\mu_{2})}^{p}+C(p,k,\|L\|)\epsilon.\]
So
\begin{equation}\label{upper}
\|Lg_{1}+\ldots+Lg_{k}\|_{L^{p}(\mu_{2})}^{p}\leq\|L\|^{p}k+C(p,k,\|L\|)\epsilon.
\end{equation}
On the other hand,
\[\|g_{1}+\ldots+g_{k}\|_{L^{p}(\mu_{1})}\geq\|f_{1}+\ldots+f_{k}\|_{L^{p}(\mu_{1})}-
\|P_{E_{1}}^{\mu_{1}}f_{1}+\ldots+P_{E_{k}}^{\mu_{1}}f_{k}\|_{L^{p}(\mu_{1})}.\]
By construction, $\mu_{1}(E_{j})<j\epsilon$ for all $j\geq 1$. Since $f_{1},\ldots,f_{k}$ take $\pm 1$ values,
\[\|P_{E_{1}}^{\mu_{1}}f_{1}+\ldots+P_{E_{k}}^{\mu_{1}}f_{k}\|_{L^{p}(\mu_{1})}\leq k^{1+\frac{1}{p}}\epsilon^{\frac{1}{p}}.\]
Therefore,
\begin{eqnarray*}
\|g_{1}+\ldots+g_{k}\|_{L^{p}(\mu_{1})}&\geq&\|f_{1}+\ldots+f_{k}\|_{L^{p}(\mu_{1})}-k^{1+\frac{1}{p}}\epsilon^{\frac{1}{p}}\\&\geq&
\|f_{1}+\ldots+f_{k}\|_{L^{2}(\mu_{1})}-k^{1+\frac{1}{p}}\epsilon^{\frac{1}{p}}=\sqrt{k}-k^{1+\frac{1}{p}}\epsilon^{\frac{1}{p}}.
\end{eqnarray*}
Thus, in view of (\ref{upper}), $L$ is not bounded below.
\end{proof}
\begin{lemma}\label{diagembed}
Let $\mu$ be a purely nonatomic measure on $\mathbb{C}$ with compact support. Let $D$ be a diagonal operator on $l^{p}$. Assume that each entry of $D$ is in the support of $\mu$. Then for every $\epsilon>0$, there exist $L:l^{p}\to L^{p}(\mu)$ and $R:L^{p}(\mu)\to l^{p}$ such that
\begin{enumerate}[(a)]
\item $\|L\|<1+\epsilon$ and $\|R\|<1+\epsilon$;
\item $M_{\mu}L-LD$ is compact and has norm at most $\epsilon$;
\item $RM_{\mu}-DR$ is compact and has norm at most $\epsilon$; and
\item $RL-I$ is compact and has norm at most $\epsilon$.
\end{enumerate}
\end{lemma}
\begin{proof}
Let $\lambda_{1},\lambda_{2},\ldots$ be the entries of $D$. For $\lambda\in\mathbb{C}$ and $\epsilon>0$, let $A_{\lambda,\epsilon}=\{z\in\mathbb{C}:|z-\lambda|<\epsilon\}$. Define
\[Lx=\sum_{n=1}^{\infty}x_{n}\frac{I(A_{\lambda_{n},\epsilon_{n}})}{\|I(A_{\lambda_{n},\epsilon_{n}})\|_{L^{p}(\mu)}},\quad x=(x_{1},x_{2},\ldots)\in l^{p}.\]
\[e_{n}^{*}(Rf)=\frac{1}{\|I(A_{\lambda_{n},\epsilon_{n}})\|_{L^{q}(\mu)}}\int_{A_{\lambda_{n},\epsilon_{n}}}f\,d\mu,\quad f\in L^{p}(\mu),\,n\geq 1,\]
where $e_{1}^{*},e_{2}^{*},\ldots$ are the coordinate functionals on $l^{p}$ and $\displaystyle\frac{1}{p}+\frac{1}{q}=1$. If $\epsilon_{n}\to 0$ fast enough, $L$ and $R$ satisfy the required conditions.
\end{proof}
\begin{lemma}\label{atomabsorb}
Let $\mu$ be a purely nonatomic measure on $\mathbb{C}$ with compact support. Let $D$ be a diagonal operator on $l^{p}$ such that each entry is in the support of $\mu$. Then $M_{\mu}$ is approximately similar to $M_{\mu}\oplus D$.
\end{lemma}
\begin{proof}
Observe that $D\oplus D\oplus\ldots$ (which is an operator acting on $(l^{p}\oplus l^{p}\oplus)_{l^{p}}$) satisfy the same conditions as $D$. By Lemma \ref{diagembed} and a routine argument (see, e.g., \cite[Lemma 2.7]{Boedihardjo}), there is a constant $C>0$ such that for every $\epsilon>0$, there exist a compact operator $K$ on $L^{p}(\mu)$, a Banach space $Y$ and $T\in B(Y)$ such that $\|K\|<\epsilon$ and $M_{\mu}+K$ is $C$-similar to $T\oplus D\oplus D\oplus\ldots$. But
\[T\oplus D\oplus D\oplus\ldots=(T\oplus D\oplus D\oplus\ldots)\oplus D\]
is $C$-similar to $(M_{\mu}+K)\oplus D$. Therefore, $M_{\mu}$ is approximately similar to $M_{\mu}\oplus D$.
\end{proof}
The proof of the following lemma is routine and thus is skipped.
\begin{lemma}\label{outlier}
Let $D_{1}$ and $D_{2}$ be diagonal operators on $l^{p}$ such that the set of entries of $D_{1}$ and the set of entries of $D_{2}$ have the same cluster points. Assume that for each $i=1,2$, all entries of $D_{i}$ are distinct. Then there exists a compact operator $K$ on $l^{p}$ such that $D_{1}$ is similar to $D_{2}+K$.
\end{lemma}
\begin{proof}[Proof of Theorem \ref{main2}]
Assume (i). For $p>2$, by the Lebesgue-Radon-Nikodym Theorem \cite[Theorem 6.10]{Rudin} and Lemma \ref{nonembed}, $\mu_{1}'$ and $\mu_{2}'$ are equivalent. For $p<2$, using duality, we also have that $\mu_{1}'$ and $\mu_{2}'$ are equivalent.

The essential spectra of $M_{\mu_{1}}$ and $M_{\mu_{2}}$ must coincide. It is easy to see that the essential spectrum of $M_{\mu_{i}}$ consists of the cluster points of $\mathrm{supp}(\mu_{i})$. So $\mathrm{supp}(\mu_{1})$ and $\mathrm{supp}(\mu_{2})$ have the same cluster points. Thus we obtain (ii).

Conversely assume (ii). Since $\mu_{1}'$ and $\mu_{2}'$ are equivalent, $M_{\mu_{1}'}$ and $M_{\mu_{2}'}$ are similar. For $i=1,2$, we write
\[M_{\mu_{i}}=M_{\mu_{i}'}\oplus T_{i},\]
where $T_{i}$ is a diagonal operator on a (finite or infinite dimensional) $l^{p}$ space whose entries are distinct and consist of atoms of $\mu_{i}$. So taking $\mu=\mu_{i}'$ in Lemma \ref{atomabsorb}, we find that $M_{\mu_{i}}$ is approximately similar to $M_{\mu_{i}'}\oplus D_{i}$, where $D_{i}$ is a diagonal operator on a $l^{p}$ space whose entries are distinct and dense in $\mathrm{supp}(\mu_{i})$. By Lemma \ref{outlier}, we obtain (i).
\end{proof}
\begin{theorem}
Let $\mu_{1}$ and $\mu_{2}$ be measures on $\mathbb{C}$ with compact supports. Let $\mu_{1}'$ and $\mu_{2}'$ be the purely nonatomic parts of $\mu_{1}$ and $\mu_{2}$. Then the following statements are equivalent.
\begin{enumerate}[(i)]
\item $M_{\mu_{1}}$ and $M_{\mu_{2}}$ are approximately similar.
\item $\mu_{1}'$ and $\mu_{2}'$ are equivalent and $\mathrm{supp}(\mu_{1})=\mathrm{supp}(\mu_{2})$.
\end{enumerate}
\end{theorem}
\begin{proof}
Assume (i). By Theorem \ref{main2}, $\mu_{1}'$ and $\mu_{2}'$ are equivalent. The spectra of $M_{\mu_{1}}$ and $M_{\mu_{2}}$ must coincide. Since the spectrum of $M_{\mu_{i}}$ is $\mathrm{supp}(\mu_{i})$ for $i=1,2$, it follows that $\mathrm{supp}(\mu_{1})=\mathrm{supp}(\mu_{2})$. Thus we obtain (ii).

Conversely assume (ii). Since $\mu_{1}'$ and $\mu_{2}'$ are equivalent, $M_{\mu_{1}'}$ and $M_{\mu_{2}'}$ are similar. In view of Lemma \ref{atomabsorb}, we may assume that $\mu_{i}$ has no atom in $\mathrm{supp}(\mu_{i}')$ for $i=1,2$. Thus $M_{\mu_{i}}$ is the direct sum of $M_{\mu_{i}'}$ and a diagonal operator $D_{i}$ on a finite or infinite dimensional $l^{p}$ space whose entries are distinct and consist of points in $\mathrm{supp}(\mu_{i})\backslash\mathrm{supp}(\mu_{i}')$.

Since $\mu_{1}'$ and $\mu_{2}'$ are equivalent and $\mathrm{supp}(\mu_{1})=\mathrm{supp}(\mu_{2})$, we have $\mathrm{supp}(\mu_{1})\backslash\mathrm{supp}(\mu_{1}')=\mathrm{supp}(\mu_{2})\backslash\mathrm{supp}(\mu_{2}')$. So $D_{1}$ and $D_{2}$ are similar. Therefore $M_{\mu_{1}}$ and $M_{\mu_{2}}$ are approximately similar.
\end{proof}
\section{Open problems}
\begin{problem}
Let $1<p<\infty$, $p\neq 2$. Let $\mu$ be the Lebesgue measure on $[0,1]$. Does there exist a compact operator $K$ on $L^{p}(\mu)$ such that $M_{\mu}\oplus M_{\mu}$ is similar to $M_{\mu}+K$?
\end{problem}
This problem has an affirmative answer when $p=2$ \cite[Corollary 2.13]{Pearcy}.
\begin{problem}
Let $1<p<\infty$, $p\neq 2$. Let $\mu$ be the Lebesgue measure on $S(0,1)$. Does there exist a 1-summing operator $K$ on $L^{p}(\mu)$ and an operator $D\in L^{p}(\mu)$ that is diagonal with respect to an unconditional basis for $L^{p}(\mu)$ such that $M_{\mu}=D+K$?
\end{problem}
This problem also has an affirmative answer when $p=2$ \cite{Voiculescu}.

\end{document}